\documentclass[11pt]{article}
\usepackage{amssymb}
\usepackage{amsfonts}
\usepackage{amsmath}
\usepackage{amscd}
\usepackage{amsthm}
\usepackage{latexsym}
\usepackage[all]{xy}
\usepackage{CJK}
\usepackage{setspace}
\usepackage{indentfirst}
\numberwithin{equation}{section}

\newcommand {\Hom}{\textrm{Hom}}
\newcommand {\Ext}{\textrm{Ext}}

\date{}

\newtheorem*{theorem1}{Theorem A}
\newtheorem*{theorem2}{Theorem B}
\newtheorem*{theorem3}{Theorem C}
\newtheorem{definition}{Definiton}
\newtheorem*{remark}{Remark}

\newtheorem{dy}{Definition}[section]

 \newtheorem*{example}{Example}

\newtheorem{dl}{Theorem}[section]
\newtheorem{co}[dl]{Corollary}
\newtheorem{yl}[dl]{Lemma}

\begin{document}

\title{\bf Some applications of $\tau $-tilting theory$^\star$}
\author{{\small Shen Li, \ Shunhua Zhang}\\
         {\small School of Mathematics, Shandong University,
        Jinan, 250100,P.R.China}}
\date{}
\maketitle
\pagenumbering{arabic}
\setlength{\parskip}{0.1\baselineskip}

\begin{abstract}
Let $A$ be a finite dimensional algebra over an algebraically closed field $k$, and $M$ be a partial tilting $A$-module. We prove that the Bongartz $\tau$-tilting complement of $M$ coincides with its Bongartz complement, and then we give a new proof of that every almost complete tilting $A$-module has at most two complements. Let $A=kQ$ be a path algebra.  We prove that the support $\tau$-tilting quiver $\overrightarrow{Q}({\rm s}\tau$-${\rm tilt} A)$ of $A$ is connected. As an application, we investigate the conjecture of Happel and Unger in [9] which claims that each connected component of the tilting quiver $\overrightarrow{Q}({\rm tilt} A)$ contains only finitely many non-saturated vertices. We prove that this conjecture is true for $Q$ being all Dynkin and Euclidean quivers and wild quivers with two or three vertices, and we also give an example to indicates that this conjecture is not true if $Q$ is a wild quiver with four vertices.
\end{abstract}

\vskip0.1in



\footnote{ $^\star$ Supported by the NSF of China (Grant Nos.11171183  and  11371165),  and also supported by PCSIRT ( IRT1264).}

\footnote{  Email addresses: \  fbljs603@163.com(S.Li), \ \    shzhang@sdu.edu.cn(S.Zhang).}

\vskip0.2in

\section{Introduction}

Adachi, Iyama and Reiten introduce $\tau$-tilting theory which completes the classical tilting theory from the viewpoint of mutation in [1],  and they establish a bijection between the tilting objects in a cluster category and the support $\tau$-tilting modules over each cluster-tilted algebra.

\vskip0.2in

As a generalization of classical tilting modules, support $\tau$-tilting modules satisfy many nice properties. For example, every basic almost complete support $\tau$-tilting module is the direct summand of exactly two basic support $\tau$-tilting modules. This means that mutation of support $\tau$-tilting modules is always possible. Moreover, the set of support $\tau$-tilting modules has a natural structure of poset and the Hasse quiver of this poset coincides with the mutation quiver of support $\tau$-tilting modules. It is also known that there are close relations between support $\tau$-tilting modules, functorially finite torsion classes and two-term silting complexes,  see [1] for details.

\vskip0.2in

In this paper, we use the properties of support $\tau$-tilting modules to prove that the Bongartz $\tau$-tilting complement of a partial tilting module coincides with its Bongartz complement, and then we give a new proof of that every almost complete tilting $A$-module has at most two complements. As an application, we prove that the support $\tau$-tilting quiver $\overrightarrow{Q}($s$\tau$-tilt$A)$ of $A$ is connected if $A$ is hereditary.  Moreover, we investigate the conjecture of Happel and Unger in [9] which claims that each connected component of the tilting quiver $\overrightarrow{Q}(tiltA)$ contains only finitely many non-saturated vertices. We prove that this conjecture is true for $Q$ being all Dynkin and Euclidean quivers and wild quivers with two or three vertices, and we also give an example to indicates that this conjecture is not true if $Q$ is a wild quiver with four vertices.

\vskip0.2in

Let $A$ be a finite dimensional algebra over an algebraically closed field $k$. For an $A$-module $M$, we denote by $|M|$ the number of pairwise nonisomorphic indecomposable  direct summands of $M$.

\vskip0.2in

 An $A$-module $T$ is called a tilting module if it satisfies the following conditions:
\begin{spacing}{1.1}
\par $(1)$ pd\,$_{A}T\leq 1$;
\par $(2)$ $\Ext^{1}_{A}(T,T)=0$;
\par $(3)$ There is a short exact sequence $0\rightarrow A\rightarrow T_{1}\rightarrow T_{2}\rightarrow 0$ with $T_{1},T_{2}\in$\,add\,$T$.
\end{spacing}

\vskip0.2in

An $A$-module $M$ satisfying the above conditions (1) and (2) is called a partial tilting module and if moreover $|M|=|A|-1$, then $M$ is called an almost complete tilting module.

\vskip0.2in

 The following definition is taken from [1].

\begin{definition}
$(a)$ An $A$-module $M$ is called $\tau$-rigid if $\Hom_{A}(M,\tau M)=0$.
\par $(b)$ An $A$-module $M$ is called $\tau$-tilting $($respectively almost complete $\tau$-tilting$)$ if $M$ is $\tau$-rigid and $|M|=|A|$ $(respectively \,|M|=|A|-1)$.
\par $(c)$ An $A$-module $M$ is called support $\tau$-tilting if there exists an idempotent $e$ in $A$ such that $M$ is a $\tau$-tilting $(A/\langle e\rangle)$-module.
\end{definition}

\vskip0.2in

From the above definition we know that any tilting (partial tilting) $A$-module $M$ is $\tau$-tilting ($\tau$-rigid). Let $M$ be a partial tilting $A$-module.
By  [1, Theorem 2.10] there exists a $\tau$-rigid $A$-modules $X$ such that $M\oplus X$ is a $\tau$-tilting $A$-module and Fac\,$(M\oplus X)=^{\perp}$$(\tau M)$. $X$ is called the Bongartz $\tau$-tilting complement of $M$. The partial tilting $A$-module $M$ also has a Bongartz complement. We prove the following theorem.

\vskip0.2in

\begin{theorem1}
Let $M$ be a partial tilting $A$-module and $X$ be its Bongartz $\tau$-tilting complement. Then pd\,$_{A}X\leq 1$ and $T=M\bigoplus X$ is a tilting $A$-module.
In particular, $X$ coincides with the Bongartz complement of $M$.
\end{theorem1}

\vskip0.2in

D.Happel and L.Unger  prove in [6] that for an almost complete tilting $A$-module $M$, it has exactly two nonisomorphic complements if and only if $M$ is faithful. In this paper, we give a new proof of this theorem from the viewpoint of mutation of support $\tau$-tilting modules.

 \vskip0.2in

 Tiling quiver $\overrightarrow{Q}$(tilt$A)$ is introduced in [15] by Riedtmann and Schofield, which gives an explicit description of relations between tilting modules. Also Adachi, Iyama and Reiten define the support $\tau$-tilting quiver $\overrightarrow{Q}({\rm s}\tau$-${\rm tilt} A)$ in [1]. We prove that the tilting quiver $\overrightarrow{Q}($tilt$A)$ can be embedded into the support $\tau$-tilting quiver $\overrightarrow{Q}({\rm s}\tau$-${\rm tilt}A)$. Then we calculate the number of arrows in $\overrightarrow{Q}({\rm tilt}A)$ when $A=kQ$ is a Dynkin hereditary algebra and show that the number of arrows in $\overrightarrow{Q}({\rm tilt} A)$ is independent of the orientation of $Q$. It is known that $\overrightarrow{Q}({\rm tilt} A)$ may not be connected when $A$ is a hereditary algebra. But for $\overrightarrow{Q}({\rm s}\tau$-${\rm tilt} A)$, we give the following result.

\vskip0.2in

\begin{theorem2}
Let $A$ be a finite dimensional hereditary algebra. Then the support $\tau$-tilting quiver $\overrightarrow{Q}({\rm s}\tau$-${\rm tilt} A)$  is connected.
\end{theorem2}

\vskip0.2in

Assume $A=kQ$ is a finite dimensional hereditary algebra. Note that the tilting quiver $\overrightarrow{Q}({\rm tilt} A)$ may contain several connected components. A conjecture of Happel and Unger in [9] is that each connected component of  $\overrightarrow{Q}({\rm tilt} A)$ contains finitely many non-saturated vertices. We prove that this conjecture is true for $Q$ being all Dynkin and Euclidean quivers and wild quivers with two or three vertices.

\vskip0.2in

\begin{theorem3}
Let $A=kQ$ be a finite dimensional hereditary algebra. If $Q$ is a Dynkin quiver, a Euclidean quiver or a wild quiver with two or three vertices, then each connected component of the tilting quiver $\overrightarrow{Q}({\rm tilt} A)$ contains finitely many non-saturated vertices.
\end{theorem3}

\vskip0.2in

\begin{remark}
Let $Q:1\leftleftarrows 2\leftarrow 3 \rightarrow4$ and $B=kQ$. We will show that the tilting quiver $\overrightarrow{Q}({\rm tilt} B)$ contains a connected component which has infinitely many non-saturated vertices.  Therefore, the conjecture of Happel and Unger is not true for some wild quivers.
\end{remark}

\vskip0.2in

This paper is arranged as follows. In section 2, we fix the notations and recall some necessary facts needed for our research. In section 3,  we prove Theorem A. Section 4 and section 5 are devoted to the proof of Theorem B and Theorem C respectively.

 \section{Preliminaries}

 Let $A$ be a finite dimensional algebra over an algebraically closed field $k$. We denote by mod-$A$ the category of all finitely generated right $A$-modules and by $D=\Hom_{k}(-,k)$ the standard duality between mod-$A$ and mod-$A^{op}$. We denote by $\tau_{A}$ the Auslander-Reiten translation of $A$.

\vskip0.2in

 Given any $A$-module $M$, Fac\,$M$ is the subcategory of mod-$A$ whose objects are generated by $M$ and add\,$M$ is the subcategory of mod-$A$ whose objects are the direct summands of finite direct sums of copies of $M$. We denote by $M^{\perp}$ $($respectively $\,^{\perp}M)$ the subcategory of mod-$A$  with objects $X\in$ mod-$A$ satisfying $\Hom_{A}(M,X)=0$(respectively $\Hom_{A}(X,M)=0$ ). pd\,$_{A}M$ is the projective dimension of $M$. We decompose $M$ as $M\cong \oplus ^{m}_{i=1}M^{d_{i}}_{i}$, where each $M_{i}$ is indecomposable, $d_{i}>0$ for any $i$ and $M_{i}$ is not isomorphic to $M_{j}$ if $i\neq j$. The module $M$ is
 called $basic$ if $d_{i}=1$ for any $i$. If $M$ is basic, we define $M[i]=\oplus_{j\neq i}M_{j}$.

\vskip0.1in

  For $\tau$-tilting modules, we have the following result in [1].

 \vskip0.1in

 \begin{yl} {\rm [1, Proposition 1.4]}
 Any faithful $\tau$-tilting $A$-module is a tilting $A$-module.
 \end{yl}

\vskip0.1in

Some certain pairs of $A$-modules are introduced in [1], and it is convenient to view $\tau$-rigid modules and support $\tau$-tilting modules as these pairs.
\begin{dy}
Let $(M,P)$ be a pair with $M \in mod$-$A$ and $P \in  proj$-$A$.
\par $(a)$ We call $(M,P)$ a $\tau$-rigid pair if $M$ is $\tau$-rigid and $\Hom_{A}(P,M)=0$.
\par $(b)$ We call $(M,P)$ a support $\tau$-tilting$\,(respectively \ $almost complete support $\tau$-tilting$)$ pair if $(M,P)$ is a $\tau$-rigid pair and $|M|+|P|=|A|(respectively \ |M|+|P|=|A|-1)$.
\end{dy}

\vskip0.1in

$(M,P)$ is called basic if $M$ and $P$ are basic and we say $(M,P)$ is a direct summand of $(M^{'},P^{'}) $ if $M$ is a direct summand of $M^{'}$ and $P$ is a direct summand of $P^{'}$. One of the main results in [1] is the following.

\vskip0.1in

\begin{yl} {\rm [1, Theorem 2.18]}
Any basic almost complete support $\tau$-tilting pair $(U,Q)$ is a direct summand of exactly two basic support $\tau$-tilting pairs $(T,P)$ and $(T^{'},P^{'})$.
\end{yl}

\vskip0.1in

Then $(T,P)$ is called left mutation of $(T^{'},P^{'})$ if Fac\,$T\subseteq$ Fac\,$T^{'}$ and this is denoted by $T=\mu^{-}(T^{'})$.  Adachi, Iyama and Reiten show in [1] that one can calculate left mutations of support $\tau$-tilting modules by exchange sequence constructed from left approximations.

\vskip0.1in

\begin{yl}  {\rm [1, Theorem 2.30]}
Let $T=X\oplus U$ be a basic $\tau$-tilting  $A$-module where the indecomposable $A$-module $X$ is the Bongartz $\tau$-tilting complement of U.  Let $X\xrightarrow{f}U^{'}\xrightarrow{g}Y\rightarrow0$ be an exact sequenceㄛ where f is a minimal left add $U$-approximation.
Then we have the following.
\par $(a)$ If U is not sincere, then Y=0. In this case U=$\mu^{-}_{X}$(T) holds and it is a basic support $\tau$-tilting A-module which is not $\tau$-tilting.
\par $(b)$ If U is sincere, then Y is a direct sum of copies of an indecomposable A-module Y$_{1}$ and $Y_{1}\notin$ add $T$. In this case $Y$$_{1}\oplus $U=$\mu^{-}_{X}(T)$ holds and it is a basic $\tau$-tilting A-module.
\end{yl}

\vskip0.1in

 The support $\tau$-tilting quiver $\overrightarrow{Q}({\rm s}\tau$-${\rm tilt} A)$ is defined  as follows:

 \vskip0.1in

\begin{dy}
$(1)$The set of vertices is ${\rm s}\tau$-${\rm tilt} A$
\par $(2)$There is an arrow from T to U if U is a left mutation of T.
\end{dy}

\vskip0.1in

Since we have a bijection $T\rightarrow$ Fac\,$T$ between basic support $\tau$-tilting modules and functorially finite torsion classes, there exists a natural partial order on the set s$\tau$-tilt$A$ of support $\tau$-tilting $A$-modules: $T_{1}<T_{2}$, if Fac\,$T_{1}\subseteq$ Fac\,$T_{2}$. Moreover, the Hasse quiver of this poset coincides with the support $\tau$-tilting quiver $\overrightarrow{Q}({\rm s}\tau$-${\rm tilt} A)$.

\vskip0.2in

The following lemma in [1] is very useful.

\vskip0.1in

\begin{yl}{\rm [1, Lemma 2.20]}
Let $(T,P)$ be a $\tau$-rigid pair for A and $P($Fac\,$T$) be the direct sum of one copy of each indecomposable Ext-projective $A$-modules in $Fac\,T$. If U is a $\tau$-rigid A-module satisfying $^{\perp}(\tau T)\cap P^{\perp}\subseteq ^{\perp}$$(\tau U)$, then there is an exact sequence $U\xrightarrow{f}T^{'}\rightarrow C \rightarrow 0$ satisfying the following conditions
\par $(1)$ f is a minimal left Fac\,$T$-approximation.
\par $(2)$ $T^{'}\in$add\,$T$, $C\in$add\,$P($Fac\,$T$) and add\,$T^{'}\bigcap\,$add\,$C$=0.
\end{yl}

\vskip0.1in

Let $A$ be a finite dimensional hereditary algebra and  $\mathcal{C}_{A}$ be the cluster category associated to $A$. We assume that $\mathcal{C}_{A}$ has a cluster-tilting object $T$ and $\Lambda $=End$_{\mathcal{C}}(T)$ is the cluster-tilted algebra. We have the following.

\vskip0.1in

\begin{yl}{\rm [1, Theorem 4.1]}
There exists a bijection between basic cluster tilting objects in $\mathcal{C}_{A}$ and the basic support $\tau$-tilting modules over $\Lambda $
\end{yl}

\vskip0.1in

Assume $A=kQ$ is a finite dimensional hereditary algebra where $Q$ is a finite quiver with $n$ vertices and $a_{s}(Q)$$(1\leq s \leq n)$  denote the number of basic support $\tau$-tilting $A$-modules with $s$ nonisomorphic indecomposable direct summands. Note that the support $\tau$-tilting $A$-modules coincide with the support tilting $A$-modules since $A$ is a hereditary algebra. If $Q$ is a Dynkin quiver, according to [13], all $a_{s}(Q)$$(1\leq s \leq n)$ are constants and do not depend on the orientation of $Q$.

\vskip0.1in

\begin{yl} {\rm [13, Theorem 1]}
Let $A=kQ$ be a path algebra of a Dynkin quiver $Q$. Then we have \\
\\
\begin{table}[h]
\centering
 \begin{tabular}{|c|c|c|c|c|c|}
\hline
$Q$ & $A_{n}$ & $D_{n}$ & $E_{6}$ & $E_{7}$& $E_{8}$\\
\hline
$a_{n}(Q)$& $\frac{1}{n+1}$$  C^{n}_{2n}$&$  \frac{3n-4}{2n-2}C^{n-2}_{2n-2}$&$418$&$2431$&$17342$\\
\hline
$a_{n-1}(Q)$&$ \frac{2}{n+1}C^{n-1}_{2n-1}$&$\frac{3n-4}{2n-3}C^{2n-3}_{n-1} $&$228$&$1001$&$4784$\\
\hline
\end{tabular}
\end{table}
\end{yl}

\vskip0.1in

Throughout this paper, we follow the standard terminologies and
notations used in the representation theory of algebras, see [3, 4, 16].

\section{Complements of partial tilting modules}

Let $A$ be a finite dimensional algebra over an algebraically closed field $k$. In this section, we prove Theorem A and give a new proof of that every almost complete tilting module has at most two complements.

\vskip0.2in

Let $M$ be a partial tilting $A$-module. It has been proved in [5] that $M$ has a complement $Y$, which is called the Bongartz complement. This complement is constructed by a universal sequence $0\rightarrow A \rightarrow E \rightarrow M^{s}\rightarrow 0$, where $s$=dim\,$_{k}\Ext^{1}_{A}(M,A)$ and $E=Y^{t}\oplus M^{'}$ with $M^{'}\in$\,add\,$M$ and some integer $t$.

\vskip0.2in

Note that $M$ is also a $\tau$-rigid $A$-module. By [1, Theorem 2.10], there exists a $\tau$-rigid $A$-module $X$ such that $T=M\oplus X$ is  $\tau$-tilting and Fac\,$T=^{\perp}$$(\tau M)$. $X$ is called the Bongartz $\tau$-tilting complement of $M$ and it is unique up to isomorphism. We prove that $X$ coincides with the Bongartz complement $Y$.

\vskip0.1in

\begin{dl}
Let $M$ be a partial tilting $A$-module and $X$ be its Bongartz $\tau$-tilting complement. Then pd\,$_{A}X\leq 1$ and $T=M\bigoplus X$ is a tilting $A$-module.
In particular, $X$ coincides with the Bongartz complement of $M$.
\end{dl}

\begin{proof}
Note that pd$_{A}M\leq1$ since $M$ is a partial tilting $A$-module. Then we have $\Hom _{A}(DA,\tau M)=0$. This implies that $DA \in ^{\perp}$$(\tau M)=$Fac\,$T$ and $T$ is faithful. By Lemma 2.1, $T$ is a tilting $A$-module and pd\,$_{A}X\leq 1$.

\par We claim that $X$ is the Bongartz complement of $M$. In fact, assume $X=\oplus^{r}_{i=1}X_{i}$ is basic and $T[i]=M\oplus X[i]$.  By [15, Proposition 1.2], we only need to show that there is no surjection from any module in add\,$T[i]$ to $X_{i}$ for $i=1,2,...,r$. If there exists such a surjection, $X_{i}$ is generated by $T[i]$ and Fac\,$T$=Fac\,$T[i]$=$^{\perp}(\tau M)$. This implies that $X[i]$ is also the Bongartz $\tau$-tilting complement of $M$, a contradiction.
\end{proof}

\vskip0.1in

\begin{remark}
\par By Lemma $2.6$, we have a short exact sequence $0\rightarrow A \xrightarrow{f}T_{1}\xrightarrow{g}T_{2}\rightarrow 0$ with $T_{1},T_{2}\in$add\,$T$ and add\,$T_{1}\,\cap$add\,$T_{2}$=0. $f$ is injective since $A$ is cogenerated by $T$. Let us show that $X\in$add\,$T_{1}$. It is obvious that all $X_{i}\in$add\,$(T_{1}\oplus T_{2})$ since $T$ is a tilting $A$-module. If there exists some $X_{i}\in$add\,$T_{2}$, then $X_{i}$ is generated by $T_{1}$ and then by $T[i]$ since add\,$T_{1} \cap$ add\,$T_{2}$=0. This contradicts the fact that $X$ is the Bongartz complement of $M$. As a result, $X\in$add\,$T_{1}$ and $T_{2}\in$add\,$M$. This short exact sequence is the universal sequence constructed in [5].
\end{remark}

\vskip0.2in

Let $M$ be an almost complete tilting $A$-module. Then $M$ has at most two complements and it has exactly two complements if and only if it is faithful (see [15, 6]). By using the mutation of support $\tau$-tilting modules, we give a new proof of these results.

\begin{dl} {\rm [6, Proposition 2.3]}
Let $M$ be an almost complete tilting $A$-module. Then $M$ has exactly two complements if it is faithful. Otherwise, it has only one complement.
\end{dl}

\begin{proof}
Let $X$ be the Bongartz complement of $M$. $(M,0)$ is an almost complete support $\tau$-tilting pair. By Lemma 2.3, it is a direct summand of exactly two support $\tau$-tilting pairs. Obviously, one is $(M\oplus X,0)$ and the other is of the form $(M\oplus Y,0)$ with $Y$ indecomposable and $M\oplus Y$ $\tau$-tilting or $(M,P)$ with $P$ projective and $\Hom_{A}(P,M)=0$. In the first case, by Lemma 2.4, there exists an exact sequence $X\rightarrow M^{'}\rightarrow Y^{s}\rightarrow 0$  with  $M^{'}\in$add$M$ and some integer $s$. Note that if a tilting $A$-module $T$ contains $M$ as a direct summand, then the support $\tau$-tilting pair $(T,0)$ contains $(M,0)$ as a direct summand. Thus $M$ has at most two complements.

\par (a) Assume $M$ is faithful. Then $M$ is sincere and $\Hom_{A}(P,M)\neq 0$ for all projective $A$-modules $P$. So the other support $\tau$-tilting pair is $(M\oplus Y,0)$  and  $M\oplus Y$ is a tilting $A$-module since it is faithful. Thus $M$ has exactly two complements $X$ and $Y$.

\par (b) Assume $M$ is not faithful. If $M$ is not sincere, then $M\oplus Y$ is not sincere since $Y$ is generated by $M$. This implies that $M\oplus Y$ is not $\tau$-tilting because all $\tau$-tilting modules are sincere. Consequently the other support $\tau$-tilting pair is $(M,P)$ and $M$ has only one complement.

\par If $M$ is sincere, the other support $\tau$-tilting pair is $(M\oplus Y,0)$. We claim that $M\oplus Y$ is not tilting. Otherwise, $A$ is cogenerated by $M\oplus Y$. Let $g:A\rightarrow F$ be an injection with $F\in{\rm add}(M\oplus Y)$. Since $Y$ is generated by $M$, there exists a surjection $h:E\rightarrow F$ with $E\in$\,add\,$M$. Since $A$ is projective there exists $f: A\rightarrow E$ with $g=hf$, hence $f$ is injective and $A$ is cogenerated by $M$, which contradicts the assumption that $M$ is not faithful. In this case $M$ has only one complement.
\end{proof}

\vskip0.2in

Let $X$ and $Y$ be two nonisomorphic complements of an almost complete tilting $A$-module $M$. It is shown in [6] that they are connected by a nonsplit short exact sequence $0\rightarrow X\xrightarrow{f}M^{'}\xrightarrow{g}Y\rightarrow0$. Now we give a different way to construct this sequence.

\begin{dl} {\rm [6, Theorem 1.1]}
Let $X$ and $Y$ be two nonisomorphic complements of an almost complete tilting $A$-module $M$ and $\Ext^{1}_{A}(Y,X)\neq 0$. Then there exists a nonsplit short exact sequence $0\rightarrow X\xrightarrow{f}M^{'}\xrightarrow{g}Y\rightarrow0$, where $f$ is a minimal left add\,$M$-approximation and $g$ is a minimal right add\,$M$-approximation.
\end{dl}

\begin{proof}
Let $X$ be the Bongartz complement of $M$. From the proof of Theorem 3.2, we know there exists an exact sequence $X\xrightarrow{f}M^{'}\xrightarrow{g} Y^{s}\rightarrow 0$ with $M^{'}\in$\,add\,$M$ and some integer $s$. Moreover, $f$ is a minimal left add\,$M$-approximation of $X$ and $g$ is a right add\,$M$-approximation of $Y^{s}$.

\par Firstly, we prove $f$ is an injection. This only needs to show $X$ is cogenerated by $M$. By the remark after Theorem 3.1, we get a short exact sequence $0\rightarrow A \rightarrow (M\oplus X)^{'}\rightarrow M^{''}\rightarrow 0$ with $(M\oplus X)^{'}\in {\rm add}(M\oplus X)$ and $M^{''}\in {\rm add} M$. Note that $M$ is faithful since it has two nonisomorphic complements. Let $\varphi :A\rightarrow F$  be an injection with $F\in {\rm add} M$. Then we have the following commutative diagram with exact rows.
\[ \xymatrix{
0\ar[r] &A \ar[d]^\varphi \ar[r] &(M\oplus X)^{'}\ar[d]^h \ar[r] &M^{''} \ar@{=}[d] \ar[r] &0\\
0\ar[r] &F \ar[r]                &E \ar[r]                       &M^{''}            \ar[r] &0}
\]
The lower sequence splits since $M$ has no self-extension, thus $E\cong F\oplus M^{''}$. Note that $\varphi$ is injective, by snake lemma, $h$ is an injection. Consequently $(M\oplus X)^{'}$ is cogenerated by $M$ and then $X$ is cogenerated by $M$.

\par Secondly, we show $g$ is right minimal, that is every $t\in$End\,$M^{'}$ such that $gt=g$ is an automorphism. Then there exists an endomorphism $\mu$ of $X$ that makes the following diagram commute. If $\mu$ is not an isomorphism, it must be nilpotent since $X$ is indecomposable and End\,$X$ is local.  So there
\[ \xymatrix{
0\ar[r] &X \ar@{.>}[d]^\mu \ar[r]^f &M^{'}\ar[d]^t \ar[r]^g &Y^{s} \ar@{=}[d] \ar[r] &0\\
0\ar[r] &X            \ar[r]^f &M^{'}         \ar[r]^g &Y^{s}            \ar[r] &0}
\]
exists some integer $m$ such that $\mu^{m}=0$. Then $t^{m}f=f\mu^{m}=0$ and so $t^{m}$ factors through $Y^{s}$, that is, there exists
$\alpha : Y^{s}\rightarrow M^{'}$ such that $t^{m}=\alpha g$. Because $gt^{m}=g$, we deduce that $g\alpha g=g$  and consequently $g\alpha =1_{Y^{s}}$ since $g$ is a surjection. This contradicts the fact that the sequence is not split. Thus $\mu$ is an isomorphism  and so is $t$.

\par Finally, we claim that $s=1$. Let $h:M_{0}\rightarrow Y$ be a minimal right ${\rm add}\ M$-approximation of $Y$ and $N={\rm Ker}\ h$. Then the map
\begin{equation*}
\psi =\left(
  \begin{array}{ccc}
   h&  & 0\\
    & \ddots &\\
    0&  & h
  \end{array}
\right) :M_{0}^{s}\longrightarrow Y^{s}
\end{equation*}
is a right add\,$M$-approximation of $Y^{s}$. Thus there is a decomposition $M_{0}^{s}=M^{'}\oplus M_{1}$ such that $\psi=(g,0)^{t}$. So there exists a map $\theta :N^{s}\rightarrow X\oplus M_{1}$ that makes the following diagram commute.
\[ \xymatrix{
0\ar[r] &N^{s} \ar@{.>}[d]^\theta \ar[r] & M_{0}^{s}\ar@{=}[d] \ar[r]^\psi &Y^{s} \ar@{=}[d] \ar[r] &0\\
0\ar[r] &X\oplus M_{1} \ar[r]^\phi                &M^{'}\oplus M_{1}\ar[r]^\psi                       &Y^{s}            \ar[r] &0}
\]
where \begin{equation*}
\phi =\left(
  \begin{array}{cc}
   f  & 0\\
    0&1
  \end{array}
\right).
\end{equation*}
\\

\par It follows that  $\theta$ is an isomorphism and $N^{s}\cong X\oplus M_{1}$. Thus we get $s=1$ since $X\notin$add\,$M_{1}$.
\end{proof}

\section{Tilting quiver and support $\tau$-tilting quiver}

Let $A$ be a finite dimensional algebra over an algebraically closed field $k$. In this section, we give a new proof of that the Hasse quiver associated to the poset of basic tilting $A$-modules coincides with the tilting quiver $\overrightarrow{Q}({\rm tilt} A)$. Moreover, when $A$ is hereditary, we calculate the number of arrows in $\overrightarrow{Q}({\rm tilt} A)$ and prove Theorem B.

\vskip0.2in

Riedtmann and Schofield define the tilting quiver $\overrightarrow{Q}({\rm tilt} A)$ in [15] as follows. The vertices are the isomorphism classes of basic tilting modules. There is an arrow $T_{1}\rightarrow T_{2}$ if $T_{1}=M\oplus X$, $T_{2}=M\oplus Y$ with $X,Y$ indecomposable and there exists a short exact sequence $0\rightarrow X \rightarrow M^{'}\rightarrow Y \rightarrow0$ with $M^{'}\in$\,add\,$M$. On the other hand, the set of basic tilting modules has a natural partial order given by $T_{1}>T_{2}$ if Fac\,$T_{1}\supseteq$Fac\,$T_{2}$. Happel and Unger have proved in [8] that the Hasse quiver associated to the poset of basic tilting modules coincides with the tilting quiver $\overrightarrow{Q}({\rm tilt} A)$.

\vskip0.2in

Note that tilting $A$-modules are also the vertices in the support $\tau$-tilting quiver  $\overrightarrow{Q}({\rm s}\tau$-${\rm tilt} A)$. Then we prove Happel and Unger's result in [8] from the viewpoint of support $\tau$-tilting modules.

\vskip0.1in

\begin{dl}{\rm [8, Theorem 4.1]}
The tilting quiver $\overrightarrow{Q}({\rm tilt} A)$ is the Hasse quiver of the poset of tilting $A$-modules.
\end{dl}

\begin{proof}
Let $T_{1}\rightarrow T_{2}$ be an arrow in $\overrightarrow{Q}({\rm tilt} A)$. Then we assume that $T_{1}=M\oplus X$, $T_{2}=M\oplus Y$ with $X,Y$ indecomposable and there exists a short exact sequence $0\rightarrow X \rightarrow M^{'}\rightarrow Y \rightarrow 0$ with $M^{'}\in$\,add\,$M$. It is obvious that Fac\,$T_{2}$=Fac\,$(M\oplus Y)\subseteq$Fac\,$M\subseteq$Fac\,$(M\oplus X)$=Fac\,$T_{1}$. Now we show the inclusion is minimal. If there exists a tilting $A$-module $T_{3}$ such that Fac\,$T_{2}\subseteq$Fac\,$T_{3}\subseteq$Fac\,$T_{1}$, then by [1, Proposition 2.26],  we have add\,$T_{1}\cap$\,add\,$T_{2}$$\subseteq$\,add\,$T_{3}$. Since add\,$ T_{1}\cap\,$add\,$T_{2}$=add\,$M$, we know $T_{3}=M\oplus X$ or  $T_{3}=M\oplus Y$.

\par Let Fac\,$T_{2}\subseteq$Fac\,$T_{1}$ be a minimal inclusion, that is  there is no tilting $A$-module $T_{3}$ $(T_{3}\ncong T_{1},\,T_{2})$ such that Fac\,$T_{2}\subseteq$Fac\,$T_{3}\subseteq$Fac\,$T_{1}$. Note that $T_{1},T_{2}\in$$\overrightarrow{Q}(s$$\tau$-$tiltA)_{0}$. Assume there exists a support $\tau$-tilting $A$-module $T$ such that Fac\,$T_{2}\subseteq$Fac\,$T\subseteq$Fac\,$T_{1}$. If $a\in A$ satisfies $a$Fac\,$T$=0, then we have $a$Fac\,$T_{2}=0$. According to  [1, Corollary 2.8], there is a bijection $T\rightarrow$Fac\,$T$ between basic tilting modules and faithful functorially finite torsion classes. Then we get $a=0$, and this implies that $T$ is a tilting $A$-module, a contradiction. Thus the inclusion Fac\,$T_{2}\subseteq$Fac\,$T_{1}$ is minimal with respect to the partial order of support $\tau$-tilting $A$-modules. As support $\tau$-tilting $A$-modules, $T_{2}$ is a left mutation of $T_{1}$ since Fac\,$T_{2}\subseteq$Fac\,$T_{1}$. By Lemma 2.4 and Theorem 3.3, there exists a short exact sequence $0\rightarrow X \rightarrow M^{'}\rightarrow Y \rightarrow0$ with $M^{'}\in$add\,$M$ and $T_{1}=M\oplus X$, $T_{2}=M\oplus Y$. It follows that there is an arrow $T_{1}\rightarrow T_{2}$ in $\overrightarrow{Q}({\rm tilt} A)$.
\end{proof}

\vskip0.1in

From the proof of Theorem 4.1 we can get the following result.

\vskip0.1in

\begin{dl}
The tilting quiver $\overrightarrow{Q}({\rm tilt} A)$ can be embedded into the support $\tau$-tilting quiver $\overrightarrow{Q}({\rm s}\tau$-${\rm tilt} A)$.
\end{dl}

\vskip0.2in

From now on, we assume that $A=kQ$ is a finite dimensional hereditary algebra. In general, the tilting quiver $\overrightarrow{Q}({\rm tilt} A)$ of $A$ may not be connected. For example, the tilting quiver $\overrightarrow{Q}({\rm tilt} A)$ is two disjoint rays when $A$ is the Kronecker algebra. However, the support $\tau$-tilting quiver $\overrightarrow{Q}({\rm s}\tau$-${\rm tilt} A)$ of $A$ is always connected.

\begin{dl}
Let $A$ be a finite dimensional hereditary algebra. Then the support $\tau$-tilting quiver $\overrightarrow{Q}({\rm s}\tau$-${\rm tilt} A)$ of $A$  is connected.
\end{dl}

\begin{proof}
Let $\overline{A}$ be the duplicated algebra of a hereditary algebra $A$ and $\overline{P}$ be the direct sum of all nonisomorphic indecomposable projective-injective $\overline{A}$-modules. For an $\overline{A}$-module $M$, we denote by $\Omega_{\overline{A}}M$ and $\Omega^{-1}_{\overline{A}}M$ respectively its first syzygy and first cosyzygy. We set $\Sigma_{1}=\{\,\Omega^{-1}_{\overline{A}}P \,|$ $P $ is an indecomposable projective A-module \}. Let $T$ be a tilting $\overline{A}$-module, we have a decomposition $T=T_{1}\oplus T_{2}\oplus \overline{P}$ with $T_{1}\in$mod-$A$ and $T_{2}\in$add\,$\Sigma_{1}$. By [2, Theorem 10], we have a bijection between tilting $\overline{A}$-modules and cluster tilting objects in $\mathcal{C}_{A}$. On the other hand, by Lemma 2.5, we get a bijection between cluster tilting objects in $\mathcal{C}_{A}$ and support $\tau$-tilting $A$-modules since $A$ is a cluster tilting object in $\mathcal{C}_{A}$.Thus there is a bijection between tilting $\overline{A}$-modules  and support $\tau$-tilting $A$-modules, sending $T=T_{1}\oplus T_{2}\oplus \overline{P}$ to $(T_{1},  \Omega_{\overline{A}}T_{2})$.

\par Then we prove there is a quiver isomorphism between $\overrightarrow{Q}({\rm tilt} \overline{A})$ and $\overrightarrow{Q}({\rm s}\tau$-${\rm tilt} A)$. It only needs to show the Hasse quiver of the poset of tilting $\overline{A}$-modules corresponds to that of support $\tau$-tilting $A$-modules.

\par Let $T$ and $T^{'}$ be tilting $\overline{A}$-modules and Fac\,$T^{'}$$\subseteq$Fac\,$T$. Then we have $ T^{'}_{1}\in$ Fac\,$(T_{1}\oplus T_{2}\oplus \overline{P})$. Since $T_{1},T^{'}_{1}\in$mod-$A$ and $T_{2}, \overline{P}\notin$mod-$A$, we get $T^{'}_{1}\in$Fac\,$T_{1}$ and then Fac\,$T^{'}_{1}\subseteq$Fac\,$T_{1}$.

\par Conversely, assume Fac\,$T^{'}_{1}\subseteq$Fac\,$T_{1}$. Since $T^{'}_{2}\notin$mod-$A$, we have $T^{'}_{2}\in$Fac\,$\overline{P}$. This implies that $ T^{'}_{1}\oplus T^{'}_{2}\oplus \overline{P}\in$Fac\,$(T_{1}\oplus T_{2}\oplus \overline{P})$ and then Fac\,$T^{'}\subseteq$Fac\,$T$.

\par According to [19, Proposition 4.1], we know that the tilting quiver $\overrightarrow{Q}({\rm tilt} \overline{A})$  of $\overline{A}$ is connected,  and hence the support $\tau$-tilting quiver $\overrightarrow{Q}({\rm s}\tau$-${\rm tilt} A)$ is connected.
\end{proof}

\vskip0.1in

\begin{example}
Let $A=kQ$ be the Kronecker algebra with $Q:1\leftleftarrows 2$. Then the support $\overrightarrow{Q}({\rm s}\tau$-${\rm tilt} A)$ is as follows.
\renewcommand\arraycolsep{0.1cm}
\newcommand{\one}{{\begin{array}{c}\vspace{-4pt}222\\\vspace{-4pt}11\\\end{array}}}
\newcommand{\two}{{\begin{array}{c}\vspace{-4pt}22\\\vspace{-4pt}1\\\end{array}}}
\newcommand{\three}{{\begin{array}{c}\vspace{-4pt}2\\\vspace{-4pt}11\\\end{array}}}
\newcommand{\four}{{\begin{array}{c}\vspace{-4pt}22\\\vspace{-4pt}111\\\end{array}}}
\newcommand{\five}{{\begin{array}{c}\vspace{-4pt}222\\\vspace{-4pt}1111\\\end{array}}}

\[\dots \rightarrow \one\oplus\two \rightarrow \two \oplus2 \rightarrow 2 \rightarrow 0 \leftarrow1 \leftarrow 1\oplus \three \rightarrow \three \oplus \four\rightarrow\dots\]
\end{example}

\vskip0.1in

Let $Q$ be a Dynkin quiver with $n$ vertices and  $A=kQ$ be the path algebra.  It is known that the number $a_{n}(Q)$ of basic tilting $A$-modules is independent of the orientation of $Q$. This implies that the number of vertices in $\overrightarrow{Q}({\rm tilt} A)$ is a constant for all Dynkin quivers of the same type.
By [12, Theorem 0.1], the number of arrows in $\overrightarrow{Q}({\rm tilt} A)$ is also a constant. By using the support $\tau$-tilting quiver $\overrightarrow{Q}({\rm s}\tau$-${\rm tilt} A)$, we give a new method to calculate the number of these arrows.

\vskip0.1in

\begin{co} {\rm [12, Theorem 0.1]}
Let $Q$ be a Dynkin quiver with $n$ vertices and $A=kQ$.  Then the number of arrows in $\overrightarrow{Q}({\rm tilt} A)$ (denoted by $\#\overrightarrow{Q}({\rm tilt} A)_{1})$ does not depend on the orientation of $Q$. In particular,
$$ \#\overrightarrow{Q}({\rm tilt} A)_{1}=\left\{
\begin{array}{ll}
C^{n+1}_{2n-1}     & \quad if\ Q=A_{n}\\
(3n-4)C^{n-3}_{2n-4} &\quad if\ Q=D_{n}    \\
1140     &\quad if\ Q=E_{6}     \\
8008      &\quad if\ Q=E_{7} \\
66976     &\quad  if\ Q=E_{8}
\end{array} \right. $$
\end{co}

\begin{proof}
We regard $\overrightarrow{Q}({\rm tilt} A)$ as a subquiver of $\overrightarrow{Q}({\rm s}\tau$-${\rm tilt} A)$. By Lemma 2.3, each vertex in $\overrightarrow{Q}({\rm s}\tau$-${\rm tilt} A)$ has exactly $n$ neighbours. Let $T$ be a tilting $A$-module, then the neighbours of $T$ in $\overrightarrow{Q}({\rm s}\tau$-${\rm tilt} A)$ are tilting $A$-modules or support $\tau$-tilting $A$-modules with $n-1$ nonisomorphic indecomposable direct summands. Note that each support $\tau$-tilting $A$-modules with $n-1$ nonisomorphic indecomposable direct summands is connected with exactly one tilting $A$-module by an arrow in $\overrightarrow{Q}({\rm s}\tau$-${\rm tilt} A)$. Then we get that
$$\#\overrightarrow{Q}({\rm tilt} A)_{1}=\frac{1}{2}(a_{n}(Q)\times n-a_{n-1}(Q)).$$
By Lemma 2.6, we can calculate the number of arrows in $\overrightarrow{Q}({\rm tilt} A)$ and this number is independent of the orientation of $Q$.
\end{proof}

\section{Non-saturated vertices in tilting quiver}

Let $Q$ be a quiver with $n$ vertices and $A=kQ$ be the finite dimensional hereditary algebra over an algebraically closed field $k$.
In this section, by using support $\tau$-tilting quiver, we give new proofs for some Happel and Unger's results. Moreover, we prove the conjecture of Happel and Unger  in [9] when $Q$ is a Dynkin quiver, a Euclidean quiver and a wild quiver with two or three vertices, and we also provide a counterexample for this conjecture when $Q$ is a wild quiver with four vertices.

\vskip0.2in

Let $T$ be a tilting $A$-module, we denote by $s(T)$(respectively $e(T)$) the number of arrows starting (respectively ending) at $T$ in the tilting quiver $\overrightarrow{Q}({\rm tilt} A)$. For a support $\tau$-tilting $A$-module $M$, by Lemma 2.3, the number of arrows starting or ending at $M$ in $\overrightarrow{Q}({\rm s}\tau$-${\rm tilt} A)$ is equal to $n$. Since $\overrightarrow{Q}({\rm tilt} A)$ can be embedded into $\overrightarrow{Q}({\rm s}\tau$-${\rm tilt} A)$, we have $s(T)+e(T)\leq n$. We say $T$ is saturated if $s(T)+e(T)=n$.

\vskip0.2in

The following result in [9] is a sufficient and necessary condition for a tilting $A$-module to be saturated in $\overrightarrow{Q}({\rm tilt} A)$. Here we give a new proof by using support $\tau$-tilting quiver.

\vskip0.1in

\begin{dl}{\rm [9,  Propostion 3.2]}
Let $T$ be a basic tilting $A$-module. It is saturated if and only if $(\underline{dim}\,T)_{i}\geq 2$ for all $1\leq i \leq n$.
\end{dl}

\begin{proof}
Assume that $T=\oplus^{n}_{i=1}T_{i}$ is saturated and there is some $i$ with $(\underline{dim}\,T)_{i}=1$. Then there must be an indecomposable summand $T_{k}$ of $T$ such that $(\underline{dim}\,T_{k})_{i}=1$. So we have a decomposition $T=T[k]\oplus T_{k}$ with $(\underline{dim}\,T[k])_{i}=0$. This implies that $T[k]$ is a non-sincere almost complete tilting $A$-module and it has only one complement. Then $T$ is not saturated, a contradiction.

\par Conversely, assume $(\underline{dim}\,T)_{i}\geq 2$ for all $1\leq i \leq n$. If $T$ is not saturated, there exists an arrow $T\rightarrow (M,P)$ in $\overrightarrow{Q}({\rm s}\tau$-${\rm tilt} A)$ where $T=M\oplus X$ with $X$ indecomposable and $P$ is an indecomposable projective $A$-module. By Lemma 2.6, we get a short exact sequence $0\rightarrow P\xrightarrow{f}T_{1}\xrightarrow{g}T_{2}\rightarrow 0$ with $T_{1}, T_{2}\in$add\,$T$ and add\,$T_{1}\cap$\,add\,$T_{2}=0$. $f$ is an injection since $P$ is cogenerated by $T$. Note that $f\neq 0$ and $\Hom_{A}(P,M)=0$, then we get $T_{1}=X^{s}\oplus M_{1}$ for some integer $s$ and $M_{1}, T_{2}\in$add\,$M$. Applying $\Hom_{A}(P,-)$ to the above short exact sequence, we get $\Hom_{A}(P,T_{1})\cong\Hom_{A}(P,P)\cong k$. This implies that $s=1$ and $(\underline{dim}\,X)_{i}=1$ for some integer $i\in (1,n)$. It is obvious that $(\underline{dim}\,M)_{i}=0$, then we have $(\underline{dim}\,T)_{i}=(\underline{dim}\,M)_{i}+(\underline{dim}\,X)_{i}=1$, a contradiction.
\end{proof}

\vskip0.1in

\begin{remark}
Let $i$ be a source vertex of $Q_{0}$ and $A=\oplus^{n}_{i=1}P_{i}$. Then we have $(\underline{dim}\oplus_{j\neq i}P_{j})_{i}$$=1$. By the above theorem, we know $A$ is not saturated. Dually, $DA$ is not saturated either.
\end{remark}

\vskip0.1in

Recall that the tilting quiver $\overrightarrow{Q}({\rm tilt} A)$ can be regarded as a subquiver of $\overrightarrow{Q}({\rm s}\tau$-${\rm tilt} A)$, then we prove the following result which is contained in [17].

\vskip0.1in

\begin{dl}{\rm [17, Theorem 3.1]}
Each connected component of the tilting quiver $\overrightarrow{Q}({\rm tilt} A)$ contains a non-saturated vertex.
\end{dl}

\begin{proof}
If $\overrightarrow{Q}({\rm tilt} A)$ is connected, then $A$ is one of the non-saturated vertices in $\overrightarrow{Q}({\rm tilt} A)$. Now assume $\overrightarrow{Q}({\rm tilt} A)$ is not connected. If $\overrightarrow{Q}({\rm tilt} A)$ has a connected component $\mathcal{R}$ which contains only saturated vertices, choose a vertex $T$ in $\mathcal{R}$. Since $\overrightarrow{Q}({\rm tilt} A)$ can be embedded into $\overrightarrow{Q}({\rm s}\tau$-${\rm tilt} A)$ which is connected, there is a path $A=T_{n}- \dots -T_{2}-T_{1}-T_{0}=T$ in the underlying graph $Q({\rm s}\tau$-${\rm tilt} A)$ where $T_{i}$  are support $\tau$-tilting $A$-modules for all $0\leq i \leq n$. Since $A$ is not contained in $\mathcal{R}$, there must exist support $\tau$-tilting $A$-modules in this path which are not tilting. Choose a minimal vertex $T_{j}$ in this path such that $T_{j}$ is a proper support $\tau$-tilting $A$-module and $T_{i}$ is tilting for all $0\leq i \leq j-1$. Note that $T_{j-1}$ is saturated since it is in $\mathcal{R}$, and this implies that the number of arrows starting or ending at $T_{j-1}$ in $\overrightarrow{Q}({\rm s}\tau$-${\rm tilt} A)$ is more than $n$, a contradiction. This completes the proof.
\end{proof}

\vskip0.1in

D.Happel and L.Unger conjecture in [9] that each connected component of $\overrightarrow{Q}({\rm tilt} A)$ contains only finitely many non-saturated vertices. Firstly we prove that this conjecture is true if $Q$ is a Dynkin or Euclidean quiver.

\vskip0.1in

\begin{dl}
Let $A=kQ$ be a finite dimensional hereditary algebra. If $Q$ is Dynkin or Euclidean type, then each connected component of $\overrightarrow{Q}({\rm tilt} A)$ contains finitely many non-saturated vertices.
\end{dl}

\begin{proof}
Let $A=kQ$ be a finite dimensional hereditary algebra. If $Q$ is a Dynkin quiver, then $A$ is a representation-finite algebra. So $\overrightarrow{Q}({\rm tilt} A)$ is finite and our result is true.

\par Assume $Q$ is a Euclidean quiver. If a tilting $A$-module $T$ is not saturated, there must be an arrow $T\rightarrow (M,P)$ in $\overrightarrow{Q}({\rm s}\tau$-${\rm tilt} A)$ where $T=M\oplus X$ with $X$ indecomposable and $P$ is an indecomposable  projective $A$-module. Then $M$ is a tilting $kQ_{i}$-module where $Q_{i}$ is a quiver obtained by removing a vertex $i$ from $Q$ and all arrows connected with $i$. Thus each non-saturated tilting $A$-module contains a tilting $kQ_{i}$-module as a direct summand. Since all path algebras $kQ_{i}$ for $1\leq i \leq n$ are representation-finite, there are only finitely many tilting $kQ_{i}$-modules. This implies that there are only finitely many non-saturated tilting $A$-modules. Then we get our result when $Q$ is a Euclidean quiver.
\end{proof}

\vskip0.1in

Before we prove this conjecture for $Q$ being a wild quiver with two or three vertices, we introduce the following lemma in [18].

\vskip0.1in

\begin{yl}{\rm [18, Main Theorem]}
Let $A=kQ$ be a finite dimensional hereditary algebra where $Q$ is a wild quiver with three vertices and $e$ be a primitive idempotent in $A$. Let regular $A$-module $M$ be a tilting $A/\langle e \rangle$-module and $M\oplus X$ be a tilting $A$-module. If $T \ncong M\oplus X$ is a vertex in the connected  component of $\overrightarrow{Q}({\rm tilt} A)$ containing $M\oplus X$, then $T$ has at least two sincere indecomposable direct summands and each sincere indecomposable direct summand of $T$ is $\tau$-sincere.
\end{yl}

\vskip0.1in

\begin{dl}
Let $\Gamma=kQ$ be a finite dimensional hereditary algebra. If $Q$ is a wild quiver with two or three vertices, then each connected component of $\overrightarrow{Q}({\rm tilt} \Gamma)$ contains finitely many non-saturated vertices.
\end{dl}

\begin{proof}
If $Q$ is a wild quiver with two vertices, it is of the form$ \xymatrix{2 \ar@/^/[r]_. \ar[r]|. \ar@/_/[r]^. & 1}$ with at least three arrows. By [14, XVIII, Corollary 2.16], there are no regular tilting $\Gamma$-modules and all tilting $\Gamma$-modules are preprojective or preinjective. The tilting quiver $\overrightarrow{Q}({\rm tilt} \Gamma)$ is of the form
\[ \circ \rightarrow \circ \rightarrow \circ \rightarrow \dots \]
\[ \circ \leftarrow \circ \leftarrow \circ \leftarrow \dots \]
It is easy to see that each  connected components of $\overrightarrow{Q}(tilt\Gamma)$  contains exactly one non-saturated vertex.
\par Assume $Q$ is a wild quiver with three vertices and $T=T_{1} \oplus T_{2}\oplus T_{3}$ is a basic tilting $\Gamma$-module. If $T$ is a non-saturated vertex in $\overrightarrow{Q}(tilt\Gamma)$, then there exists an arrow $T \rightarrow (T_{1}\oplus T_{2}, P)$ in $\overrightarrow{Q}({\rm s}\tau$-${\rm tilt} \Gamma)$ where $P$ is an indecomposable projective $\Gamma$-module.  Let $e$ be the primitive idempotent in $\Gamma$ corresponding to $P$. Then $T_{1}\oplus T_{2}$ is a tilting $\Gamma / \langle e \rangle$-module and each non-saturated tilting $\Gamma$-module contains a tilting  $\Gamma / \langle e \rangle$-module as a direct summand.

\par If $\Gamma / \langle e \rangle$ is a representation-finite algebra, we can find only finitely many non-saturated tilting $\Gamma$-modules which contain tilting $\Gamma / \langle e \rangle$-modules as direct summands.

\par If $\Gamma / \langle e \rangle$ is a representation-infinite algebra, the quiver of $\Gamma / \langle e\rangle$ is of the form$ \xymatrix{\circ \ar@/^/[r]_. \ar@{.>}[r]|. \ar@/_/[r]^. & \circ}$ with at least two arrows.
Since there are only finitely many non-sincere indecomposable preprojective and preinjective $\Gamma$-modules, all but finitely many tilting $\Gamma / \langle e \rangle$-modules are regular $\Gamma$-modules. Thus all but finitely many non-saturated tilting $\Gamma$-modules contain tilting $\Gamma / \langle e \rangle$-modules which are regular $\Gamma$-modules as direct summands. Assume $T_{1}\oplus T_{2}$ is a regular $\Gamma$-module. By Lemma 5.4, $T$ is contained in a connected component of $\overrightarrow{Q}(tilt\Gamma)$ which has only one non-saturated vertex $T$. This completes our proof.
\end{proof}

\vskip0.1in

\begin{remark}
We should mention that the conjecture of Happel and Unger is not true for some wild quivers.
\end{remark}

\vskip0.2in

In order to provide a counterexample, we need the following lemma.

\begin{yl}{\rm [17, Theorem 1]}
Let $M$ be a partial tilting $A$-module with $n-2$ nonisomorphic indecomposable direct summands and $\overrightarrow{Q}(tilt_{M}A)$ be the subquiver of $\overrightarrow{Q}(tiltA)$ with vertices $T$ such that $M$ is a direct summand of  $T$. If $M$ is not sincere and $\overrightarrow{Q}(tilt_{M}A)$ is infinite, then $\overrightarrow{Q}(tilt_{M}A)$ is of the form
\[ \circ \rightarrow \circ \rightarrow \circ \rightarrow \dots \]
\[ \circ \leftarrow \circ \leftarrow \circ \leftarrow \dots \].
\end{yl}

\vskip0.2in

The following example is taken from [17] which is a counterexample to the conjecture of Happel and Unger.

\newcommand{\six}{{\begin{array}{c}\vspace{-4pt}3\\\vspace{-4pt}4\\\end{array}}}

\begin{example}
Let $B=kQ$ be the path algebra of the wild quiver $Q:1\leftleftarrows 2\leftarrow 3 \rightarrow4$.  We claim that the tilting quiver $\overrightarrow{Q}({\rm tilt} B)$ contains a connected component which has infinitely many non-saturated vertices.
\end{example}

\vskip0.1in

Indeed, we assume that $N$ is a tilting module over the Kronecker algebra $k(1\leftleftarrows 2)$ and it has no nonzero projective direct summands.
Let $I_{3}=3$ and $I_{4}=\six$. Then $I_{3} \oplus I_{4}\oplus\tau_{B}N$ is a tilting $B$-module. The Coxeter
matrix of $B$ is
\begin{equation*}
\Phi_{B}=\left(
  \begin{array}{cccc}
    -1 & 2 & 0 & 0\\
    -2 & 3 & 1 & 0\\
    -2 & 3 & 1 & -1\\
     0 & 0 & 1 & -1
  \end{array}
\right)
\end{equation*}
By $\underline{dim}\,\tau_{B} N=\Phi_{B}\,\underline{dim}\,N$, we know $(\underline{dim}\,\tau_{B} N)_{4}=0$. Thus we get that $(\underline{dim}\,I_{4}\oplus I_{3}\oplus \tau_{B}N)_{4}=1$ and $I_{4}\oplus I_{3}\oplus\tau_{B}N$ is not saturated. Since there are infinitely many tilting modules over the Kronecker algebra, by Lemma 5.6, at least one of the connected components in $\overrightarrow{Q}({\rm tilt}_{I_{3}\oplus I_{4}}B)$ contains infinitely many non-saturated vertices and we denote this component by $\mathcal{G}$. Then the connected component in $\overrightarrow{Q}({\rm tilt} B)$ which contains $\mathcal{G}$ has infinitely many non-saturated vertices.

\section*{Acknowledgements}
The authors are grateful to the valuable discussions with Hongbo Yin.

\end{document}